\documentclass[reqno,11pt]{amsart}
\usepackage[T1]{fontenc}
\usepackage{graphicx}
\usepackage{mathrsfs}
\usepackage{amsmath,amsthm,amssymb,amscd}
\usepackage{latexsym}
\usepackage{color}
\usepackage[colorlinks=true]{hyperref}
\hypersetup{urlcolor=blue,citecolor=blue,linkcolor=blue}

\newtheorem{theorem}{Theorem}[section]
\newtheorem{corollary}[theorem]{Corollary}
\newtheorem{lemma}[theorem]{Lemma}

\newtheorem{remark}[theorem]{Remark}
\newtheorem{example}[theorem]{Example}

\theoremstyle{remark}
\numberwithin{equation}{section}

\newcommand{\rank}{\operatorname{rank}}
\newcommand{\diag}{\operatorname{diag}}
\renewcommand{\span}{\operatorname{span}}

\newcommand{\Z}{\mathbb Z}
\newcommand{\C}{\mathbb C}
\newcommand{\F}{\mathbb F}
\newcommand{\K}{\mathbb K}
\newcommand{\M}{\mathrm M}

\newcommand{\GL}{\mathrm{GL}}

\newcommand{\mtx}[1]{\begin{pmatrix}#1\end{pmatrix}}
\newcommand{\wt}[1]{\widetilde{#1}}

\newcommand{\col}{\operatorname{col}}
\newcommand{\row}{\operatorname{row}}
\newcommand{\Alt}{\operatorname{Alt}}
\newcommand{\1}{\mathbf 1}

\DeclareMathOperator{\Ima}{Im}
\begin{document}
	
	\title{Linear Rank Nonincreasing and Preserving Maps between Matrix Spaces}
	
	\author[H. Huang] {Huajun Huang}
	\address{Department of Mathematics and Statistics, Auburn University,
		AL 36849--5310, USA}
	\email{huanghu@auburn.edu}

	\author[M. Tsai] {Ming-Cheng Tsai}
	\address{General Education Center,
		Taipei University of Technology, Taipei 10608, Taiwan}
	\email{mctsai2@mail.ntut.edu.tw}

	\begin{abstract}
		We classify linear maps $\phi$ between possibly different rectangular
		matrix spaces over a field $\mathbb F$ with $|\mathbb F|\ge5$
		such that $\operatorname{rank} \phi(A)\le2$ whenever $\operatorname{rank} A\le2$.
		Every such map has one of six forms.
		If its image contains a matrix of rank greater than two, then
		$\phi(A)=MAN$ or $\phi(A)=MA^{\top}N$ for fixed matrices $M,N$.
		The remaining forms have images consisting entirely of matrices
		of rank at most two, including an exceptional form arising from
		$3\times3$ alternating matrices.
		As a consequence, we characterize linear maps preserving rank
		exactly two.
		We also prove that rank-$k$ preservation implies rank-$k$
		nonincrease when $|\mathbb F|>k+2$, and that rank-$k$
		nonincrease implies rank-$\ell$ nonincrease for every admissible
		$\ell\ge k$ when $|\mathbb F|>k+1$.
		The classification uses the fact that the Chow rank of the
		generic $3\times3$ determinant is five, connecting the rank
		preserver problem with product decompositions of polynomials.
	\end{abstract}
	
	
	
	\maketitle
	
	\section{Introduction}
	
	Linear preserver problems concern linear maps that preserve specified sets, relations, or invariants, such as the determinant, spectrum, trace, and invertibility; see \cite{Li01,Li92,Pierce,ZTC2007}. Rank conditions occupy a central place in this theory. The rank filtration of a matrix space is described by determinantal varieties: the matrices of rank at most $k$ are cut out by the $(k+1)\times(k+1)$ minors. This leads to two related questions for linear maps: which maps preserve a specified rank, and which carry matrices of rank at most $k$ into matrices of rank at most $k$? 
	
	Classical rank-preserver results concern operators on a fixed matrix space. Marcus and Moyls \cite{Ma59,MaMo59} characterized rank-preserving maps on $\M_{m,n}(\F)$ under suitable assumptions on the field. These maps have the forms
	\[
	A\longmapsto MAN,
	\qquad
	A\longmapsto MA^{\top}N,
	\]
	where the transpose form requires compatible dimensions. Later work addressed preservation of a specified rank; see \cite{We67,Mi77,Beasley88,Beasley90,Beasley01}.
	
	Rank preservation at a single level is different from nonincrease on all matrices below that level. Lim \cite{Lim} studied rank-one nonincrease on a fixed matrix space, and Loewy \cite{Lo93,Lo93II, Loewy} considered higher ranks. We allow the source and target to have different dimensions and study linear maps
	\[
	\phi:\M_{m,n}(\F)\longrightarrow\M_{p,q}(\F).
	\]
	Botta \cite{Bo87} classified rank-one nonincreasing maps in this setting, and Li, Rodman, and \v{S}emrl \cite{CKLi} characterized rank-one preservers. The present paper addresses rank two and the relations among rank conditions at different levels.
	
	Let \(\Z_+\) be the set of positive integers, let \([n]=\{1,\ldots,n\}\), and let \(\F\) be a field with \(|\F|\ge5\), unless a different field assumption is stated explicitly. No restriction on the characteristic is imposed. Write \(\M_{m,n}(\F)\) for the space of \(m\times n\) matrices over \(\F\), and write \(\M_n(\F)=\M_{n,n}(\F)\).
	
	For a set \(S\) of matrices, let \(S^{\mathrm{rk}=k}\) and \(S^{\mathrm{rk}\le k}\) denote the matrices in \(S\) of rank \(k\) and of rank at most \(k\), respectively. A linear map \(\phi:V\to W\) is a \emph{rank-\(k\) preserver} if
	\begin{equation}\label{rank k preserver}
		\phi(V^{\mathrm{rk}=k})\subseteq W^{\mathrm{rk}=k}.
	\end{equation}
	It is \emph{rank-\(k\) nonincreasing} if
	\begin{equation}\label{rank k decreasing}
		\phi(V^{\mathrm{rk}\le k})\subseteq W^{\mathrm{rk}\le k}.
	\end{equation}
	A map is rank preserving, or rank nonincreasing, if it has the stated property for every possible rank.

	We establish three main results. Theorem~\ref{thm: preserving to decreasing} shows that a rank-$k$ preserver is rank-$k$ nonincreasing when $|\F|>k+2$. Theorem~\ref{thm: rank nonincreasing lower to higher} shows that rank-$k$ nonincrease implies rank-$\ell$ nonincrease for every admissible $\ell\ge k$ when $|\F|>k+1$. In particular, both conclusions hold over every infinite field, in arbitrary characteristic. These results describe how a condition at one rank constrains the map at other ranks. Examples in Section~4 show that a map preserving a given rank need not preserve any other positive rank, and that a map that is
	rank-nonincreasing on matrices of a given rank need not have this
	property on matrices of a lower rank. 
	
	The principal result, Theorem~\ref{thm: M_mn rank 2 nonincreasing}, gives six possible forms for every rank-two nonincreasing map $\phi$ between matrix spaces when $|\F|\ge 5$. These six forms establish a sharp structural dichotomy: either $\phi$ is a standard multiplication or transpose-multiplication map $\phi(A)=MAN$ or $\phi(A)=MA^\top N$ (and hence is rank-$k$ nonincreasing for every $k\ge1$ and  $\Ima\phi$ contains a matrix of rank at least three), or the entire image $\Ima\phi$ has upper rank at most two. In the latter case, $\phi$ falls into one of four classical upper-rank-two structures \cite{AtkinsonLloyd1981,Atkinson1983}: its outputs share a common row or column space of dimension at most two, are sums of two rank-one terms, or lie in a space equivalent to an embedded $3\times3$ alternating matrix space.  Corollary~\ref{thm: Mmn rank 2 preserver} identifies the additional conditions under which the listed forms preserve rank two. 
	Examples in Section 4 indicates that none of the six forms can be fully covered by the other five.

	A notable feature of the proof is its use of the Chow rank, also called the product rank, of a polynomial. For a homogeneous cubic $f$, this is the least integer $r$ for which $f$ is a sum of $r$ products of three linear forms. The generic $3\times3$ determinant, and hence any generic $3\times3$ minor, has Chow rank five over every field: the lower bound follows from
	Ilten and Teitler \cite{IltenTeitler2016} and 
	Ilten and S\"u\ss\ \cite[Example~4.2]{IltenSuess2020}, and a characteristic-free five-term decomposition follows from Krishna and Makam \cite[Section~2]{KrishnaMakam2021}. Thus such a minor cannot be expressed as a sum of four products of linear forms. In Lemma~\ref{lem:cubic-obstruction}, a cubic polynomial forced by the rank condition has a four-term product decomposition. Its membership in the ideal generated by the $3\times3$ minors then forces it to vanish identically. This obstruction connects the linear preserver problem to determinantal ideals, the geometry of Fano schemes, and tensor decomposition. 

	The results leave two questions. One is to determine the rank-two
	classification over the remaining fields $\F_2$, $\F_3$, $\F_4$. Strengthening Theorem 4.1 to $|\F| \ge 4$ is an easy upgrade. Another, suggested by the proof of our classification theorem,
	is to investigate whether lower bounds for the Chow rank of the generic
	$(k+1)\times(k+1)$ determinant can provide analogous obstructions in
	the classification of rank-$k$ nonincreasing linear maps for $k\ge3$.
	
	Section~2 proves the implications among rank conditions. Section~3 recalls the rank-one classifications. Section~4 states the rank-two classification and its preserver corollary, gives examples distinguishing rank conditions, and then presents the auxiliary lemmas and proofs.
	
	\section{Relations among rank nonincreasing and preserving maps}\label{sec:rank-relations}

	We establish two relations among rank conditions under explicit field-size assumptions. Preservation at rank $k$ implies nonincrease at rank $k$, and nonincrease at rank $k$ extends to every admissible higher rank. Beasley proved the first implication for operators on a fixed matrix space over an algebraically closed field \cite[Lemma~2.1]{BEASLEY70}. We allow different source and target spaces and give sufficient bounds for finite fields.
	
	We will use the following elementary polynomial fact. If a polynomial over $\F$ has degree at most $d$ in each variable and $|\F|>d$, then vanishing at every point of $\F^N$ implies that it is the zero polynomial. This follows by induction on $N$ from the bound on the number of roots of a nonzero univariate polynomial. Consequently, if a polynomial $h$ of total degree at most $d$ in the entries of a matrix vanishes on all matrices over $\F$ of rank at most $s$, and $|\F|>d$, then it vanishes on such matrices over every extension field. Indeed, substitute
	\[
	X=UV^\top,\qquad U\in\M_{m,s}(\F),\quad V\in\M_{n,s}(\F).
	\]
	The polynomial $h(UV^\top)$ has degree at most $d$ in each scalar entry of $U$ and $V$, so it is identically zero. Every matrix of rank at most $s$ over an extension field admits the same factorization.
	
	\begin{theorem}\label{thm: preserving to decreasing}
		Let $\F$ be a field with $|\F|>k+2$ and let
		\(
		1\le k\le \min\{m,n,p,q\}.
		\)
		Every rank-$k$ preserving linear map
		\(
		\phi:\M_{m,n}(\F)\longrightarrow \M_{p,q}(\F)
		\)
		is rank-$k$ nonincreasing. That is,
		\[
		\rank X=k \implies \rank\phi(X)=k
		\]
		implies
		\[
		\rank X\le k \implies \rank\phi(X)\le k.
		\]
	\end{theorem}
	
	\begin{proof}
		Suppose, to the contrary, that $\phi$ is rank-$k$ preserving but
		not rank-$k$ nonincreasing. Then there exists
		$A\in\M_{m,n}(\F)$ such that
		\[
		r:=\rank A<k
		\qquad\text{and}\qquad
		\rank\phi(A)>k.
		\]
		By matrix equivalence, we can write 
		\[
		A=P
		\begin{bmatrix}
			I_r&0\\
			0&0
		\end{bmatrix}_{m\times n}
		Q
		\]
		for some nonsingular matrices
		$P\in\M_m(\F)$ and $Q\in\M_n(\F)$. 
		Choose
		\[
		B=P
		\begin{bmatrix}
			0&0&0\\
			0&I_{k-r}&0\\
			0&0&0
		\end{bmatrix}_{m\times n}
		Q,
		\]
		where the rows and columns of the middle matrix are partitioned into blocks of sizes \((r,k-r,m-k)\) and \((r,k-r,n-k)\), respectively.
		Then, for every $x\in\F\setminus\{0\}$,
		\[
		\rank(A+xB)=k.
		\]
		Since $\phi$ is rank-$k$ preserving and linear,
		\[
		\rank\phi(A+xB)=\rank\big(\phi(A)+x\phi(B)\big)=k
		\qquad
		(x\in\F\setminus\{0\}).
		\]
		Every $(k+1)\times(k+1)$ minor of
		$\phi(A)+x\phi(B)$ is a polynomial in $x$ of degree at most $k+1$.
		Since $\rank\phi(A+xB)=k$, each such polynomial vanishes for every
		$x\in\F\setminus\{0\}$. Since $\F$ has more than $k+1$ nonzero elements, it must
		vanish identically. Evaluating at $x=0$ shows that every
		$(k+1)\times(k+1)$ minor of $\phi(A)$ is zero. Hence
		\[
		\rank\phi(A)\le k,
		\]
		contrary to the choice of $A$. Therefore $\phi$ is rank-$k$
		nonincreasing.
	\end{proof}
	
	Conceptually, over an infinite field, the rank-\(k\) matrices are Zariski dense in the determinantal variety \(\{X:\operatorname{rank}X\le k\}\). Thus the vanishing of all \((k+1)\)-minors of \(\phi(X)\) on rank-\(k\) matrices propagates automatically to all matrices of rank at most \(k\).
	
	The converse of Theorem \ref{thm: preserving to decreasing} does not hold. For example, Botta's classification of rank-one nonincreasing linear maps \cite{Bo87} (see Theorem \ref{thm: rank decreasing one gen}) shows that a rank-one nonincreasing linear map need not be a rank-one preserver.

	Loewy \cite{Loewy} announced without proof that every rank-$k$ nonincreasing linear operator $\phi: \M_{m,n}(\F)\to \M_{m,n}(\F)$ over an infinite field $\F$ is also 
	rank-$\ell$ nonincreasing for every $k\le \ell\le \min\{m,n\}.$  
	We extend this result to maps $\M_{m,n}(\F)\to \M_{p,q}(\F)$ and provide a proof here.
	
	\begin{lemma}[No rank-jump lemma]\label{lem:singular-input-bound}
		Let $r\ge2$, let $\F$ be a field with $|\F|>r$, and let $\Psi:\M_r(\F)\to\M_{p,q}(\F)$ be linear. If
		\[
		\rank X<r\quad\Longrightarrow\quad\rank\Psi(X) <r,
		\]
		then $\rank\Psi(X)\le r$ for every $X\in\M_r(\F)$.
	\end{lemma}

	\begin{proof}
		If $\min\{p,q\}\le r$, the conclusion is immediate.
		Assume that $p,q\ge r+1$. Let $\K$ be an algebraic closure of $\F$, and extend $\Psi$ by the same coefficients to a $\K$-linear map
		\[
		\Psi_{\K}:\M_r(\K)\longrightarrow\M_{p,q}(\K).
		\]
		Every $r\times r$ minor of $\Psi(X)$ has degree at most $r$ and vanishes on matrices of rank at most $r-1$ over $\F$. Since $|\F|>r$, the polynomial argument at the beginning of this section, with $s=r-1$ and $d=r$, shows that these minors also vanish on all singular matrices over $\K$. Thus
		\[
		\rank X<r\quad\Longrightarrow\quad\rank\Psi_{\K}(X)<r
		\qquad(X\in\M_r(\K)).
		\]
		
		Work now in $\K[x_{ab}]$ and put $D(X)=\det X$. The ideal of the singular determinantal hypersurface over the algebraically closed field $\K$ is the prime ideal $(D)$; see  \cite[Theorem 2.10]{BrunsVetter1988}. Hence every polynomial vanishing on that hypersurface is divisible by $D$. In particular, every $r\times r$ minor of $\Psi_{\K}(X)$ is divisible by $D$.
		
		Now let $f$ be an $(r+1)\times(r+1)$ minor of $\Psi_{\K}(X)$.
		It also vanishes on singular matrices over $\K$. Since $\Psi_{\K}$ is
		linear, $f$ is homogeneous of degree $r+1$ unless it is zero.
		Thus
		\begin{equation}\label{eq:large-minor-factor}
			f(X)=D(X)L(X)
		\end{equation}
		for some homogeneous linear polynomial $L$, possibly zero.
		
		By the chain rule for the determinant, each formal  partial derivative
		$\partial f/\partial x_{ab}$ is a constant linear combination
		of $r\times r$ minors of $\Psi_{\K}(X)$. Hence
		\[
		D\mid\frac{\partial f}{\partial x_{ab}}.
		\]
		Differentiating \eqref{eq:large-minor-factor} formally gives
		\[
		\frac{\partial f}{\partial x_{ab}}
		=
		\frac{\partial D}{\partial x_{ab}}L
		+D\frac{\partial L}{\partial x_{ab}},
		\]
		and therefore
		\[
		D\mid\frac{\partial D}{\partial x_{ab}}L.
		\]
		Since $D$ is irreducible and the cofactor $\partial D/\partial x_{ab}$ is a nonzero
		polynomial of degree $r-1$,
		it follows that $D\mid L$.
		As $\deg L\le1<r$, we obtain $L=0$.
		
		Thus every $(r+1)\times(r+1)$ minor of $\Psi_{\K}(X)$ is identically zero, so $\rank\Psi_{\K}(X)\le r$ for all $X\in\M_r(\K)$. Restricting to matrices over $\F$ proves the conclusion.
	\end{proof}

	\begin{theorem}\label{thm: rank nonincreasing lower to higher}
		Let $\F$ be a field with $|\F|>k+1$, let $1\le k\le\min\{m,n\}$, and let
		$$\phi:\M_{m,n}(\F)\to\M_{p,q}(\F)$$ be rank-$k$ nonincreasing.
		Then $\phi$ is rank-$\ell$ nonincreasing for every
		$\ell\in\{k+1,\ldots,\min\{m,n\}\}$.
	\end{theorem}
	
	\begin{proof}
		If $k=\min\{m,n\}$, there is nothing to prove. Let $\K$ be an algebraic closure of $\F$, and extend $\phi$ by the same coefficients to a $\K$-linear map
		\[
		\phi_{\K}:\M_{m,n}(\K)\longrightarrow\M_{p,q}(\K).
		\]
		Every $(k+1)\times(k+1)$ minor of $\phi(X)$ has degree at most $k+1$ and vanishes whenever $\rank X\le k$. Since $|\F|>k+1$, the polynomial argument at the beginning of this section, with $s=k$ and $d=k+1$, shows that the same minors of $\phi_{\K}(X)$ vanish whenever $\rank X\le k$. Thus $\phi_{\K}$ is rank-$k$ nonincreasing.
		
		We prove by induction on $r\ge k$ that $\phi_{\K}$ is rank-$r$ nonincreasing. Assume that $\phi_{\K}$ is rank-$(r-1)$ nonincreasing, where $k<r\le\min\{m,n\}$, and let $A\in\M_{m,n}(\K)$ have rank $r$. Write $A=UV^\top$, with $U\in\M_{m,r}(\K)$ and $V\in\M_{n,r}(\K)$ of full column rank, and define
		\[
		\Psi(X)=\phi_{\K}(UXV^\top),\qquad X\in\M_r(\K).
		\]
		If $\rank X<r$, then $\rank(UXV^\top)=\rank X\le r-1$, so $\rank\Psi(X)\le r-1$. The field $\K$ is infinite, and therefore $|\K|>r$. Lemma~\ref{lem:singular-input-bound}, applied over $\K$, gives
		\[
		\rank\phi_{\K}(A)=\rank\Psi(I_r)\le r.
		\]
		Together with the induction hypothesis for inputs of smaller rank, this proves rank-$r$ nonincrease. Induction and restriction back to $\F$ give the result. All induction steps take place over $\K$. No field-size condition beyond $|\F|>k+1$ is needed.
	\end{proof}
	
	In Subsection~\ref{sec:example}, we give examples showing that a linear map preserving one rank need not preserve another, and that a rank-$k$ nonincreasing linear map need not be rank-$\ell$ nonincreasing for $\ell<k$.
	
	\section{Rank-one  nonincreasing and preserving maps}

	\subsection{Rank-one nonincreasing maps}
	
	The multiplication forms are fundamental in classical rank-preserver theory; see \cite{MaMo59,We67,Mi77}. For an invertible rank-one preserving operator on a fixed matrix space, they are
	\begin{equation}\label{matrix product forms}
		\phi(A)=MAN,\quad \text{or}\quad \text{$m=n$ and $\phi(A)=MA^\top N$,}
	\end{equation}
	where $M\in\GL_m(\F)$ and $N\in\GL_n(\F)$. The invertibility assumption matters: rank-one preservation alone can also admit maps whose range has upper rank one. For higher ranks, Beasley and Laffey \cite{Beasley90} studied preservation of a specified rank, and Loewy \cite{Lo93, Lo93II} studied rank-$k$ nonincrease under the hypotheses of those papers. We recall below the rank-one results used in our proofs.
	
	More generally, Botta \cite{Bo87} classified rank-one nonincreasing linear maps between matrix spaces of different sizes over any field $\F$. We recall the result. 
	
	\begin{theorem} [Botta  \cite{Bo87}]
		\label{thm: rank decreasing one gen}
		Suppose $m,n,p,q\in\Z_+$. A linear map $\phi: \M_{m,n}(\F)\to \M_{p,q}(\F)$ is 
		rank-one nonincreasing 
		if and only if $\phi$ takes one of the following forms:  
		\begin{enumerate}
			\item There exist $M\in \M_{p,m}(\F)$ and $N\in \M_{n,q}(\F)$  such that $\phi(A)=MAN$. 
			
			\item There exist $M\in \M_{p,n}(\F)$ and $N\in \M_{m,q}(\F)$  such that $\phi(A)=MA^\top  N$. 
			
			\item  There exist   a linear map $g:\M_{m,n}(\F)\to \F^q$ and a vector $u\in \F^p$  such that
			$\phi(A)=ug(A)^\top $. 
			\item There exist    a linear map $f:\M_{m,n}(\F)\to \F^p$ and  a vector $v\in \F^q$ such that
			$\phi(A)=f(A) v^\top $. 
			
		\end{enumerate}
	\end{theorem}
	
	\begin{remark}
		The four forms in Theorem~\ref{thm: rank decreasing one gen} can overlap. The following are examples of
		rank-one nonincreasing linear maps $\phi: \M_2(\F)\to \M_2(\F)$ that have more than one of these forms.   
		\begin{enumerate}
			\item Let  
			\begin{equation}
				\phi\left(\mtx{a_{11} &a_{12}\\a_{21} &a_{22}}\right)=\mtx{1&2\\1&2}\mtx{a_{11} &a_{12}\\a_{21} &a_{22}}=\mtx{1\\1}\mtx{a_{11}+2a_{21} &a_{12}+2a_{22}}.
			\end{equation}
			Thus $\phi$ has both forms (1) and (3).
			\item Let $\psi: \M_{2}(\F)\to \F$ be a linear functional, $C=uv^\top  \in \M_2(\F)$ has rank one where $u, v\in\F^2$, and 
			\begin{equation} 
				\phi(A)=\psi(A) C=u\psi(A) v^\top .
			\end{equation}
			Thus $\phi$ has forms   (3)  and (4).
		\end{enumerate}
	\end{remark}

	\subsection{Rank-one preservers}
	
	Li, Rodman, and \v{S}emrl \cite[Theorem 2.1]{CKLi} characterized linear rank-one preservers as follows.
	
	\begin{theorem} [Li, Rodman, \v{S}emrl \cite{CKLi}]
		\label{thm: rank one gen}
		Every linear rank-one preserver $\phi: \M_{m,n}(\F)\to \M_{p,q}(\F)$  
		takes one of the following forms for certain nonsingular  matrices $M\in \GL_{p}(\F)$ and $N\in \GL_{q}(\F)$:  
		\begin{enumerate}
			\item[(1)] $m\leq p$, $n\leq q$, $\phi(A)=M\mtx{A &0\\0&0}N$.  
			\item[(2)] $m\leq q$, $n\leq p$, $\phi(A)=M\mtx{A^\top  &0\\0&0}N$.
			\item[(3)] $\phi(A)=M\mtx{\psi(A) & 0}N$, where $\psi:\M_{m,n}(\F)\to \M_{p,1}(\F)$ is a linear map such that $\psi(A)\neq 0$  for all  $A\in \M_{m,n}(\F)$  and $\rank A=1$.
			\item[(4)] $\phi(A)=M\mtx{\psi(A) \\ 0}N$, where $\psi:\M_{m,n}(\F)\to \M_{1,q}(\F)$ is a linear map such that $\psi(A)\neq 0$  for all  $A\in \M_{m,n}(\F)$    and $\rank A=1$.
		\end{enumerate}
	\end{theorem}
	
	Theorem \ref{thm: rank one gen} implies the following characterization of linear rank preservers.
	
	\begin{corollary}
		A linear map $\phi: \M_{m,n}(\F)\to \M_{p,q}(\F)$ is a rank-$k$ preserver
		for every $k\in [\min\{m,n\}]$ if and only if $\phi$ takes one of the following forms for certain nonsingular  matrices $M\in \GL_{p}(\F)$ and $N\in \GL_{q}(\F)$:  
		\begin{enumerate}
			\item[(1)] $m\leq p$, $n\leq q$, $\phi(A)=M\mtx{A &0\\0&0}N$.   
			\item[(2)] $m\leq q$, $n\leq p$, $\phi(A)=M\mtx{A^\top  &0\\0&0}N$.
		\end{enumerate}
	\end{corollary}
	
	\begin{proof}
		$\phi$ must take one of the four forms in Theorem \ref{thm: rank one gen}.  
		
		First, suppose that $\phi$ takes the form given in Theorem~\ref{thm: rank one gen}\,(3). 
		Then $\rank \phi(A)\le 1$. By assumption, $\phi$ preserves every possible rank, so that $m=1$ or $n=1$. 
		\begin{itemize}
			\item 
			When $m=1$, the linear map $\psi:\M_{m,n}(\F)\to \M_{p,1}(\F)$ can be expressed as
			\[
			\psi(A)=KA^\top
			\]
			for some $K\in \M_{p,n}(\F)$. Since $\phi$ is a  rank-$k$ preserver, 
			$\ker\psi=\{0\}$, so $K$ has full column rank $n$, and thus $p\ge n$. 
			The columns of $K$ can be extended to a basis, yielding a nonsingular matrix 
			$K_1=\mtx{K & K'}\in \M_p(\F)$. 
			There exist nonsingular  matrices $M_1\in \M_p(\F)$ and 
			$N_1\in \M_q(\F)$ such that
			\[
			\phi(A)=M_1\mtx{\psi(A) & 0}N_1
			=M_1K\mtx{A^\top & 0}N_1
			= M_1K_1\mtx{A^\top & 0\\ 0 & 0}_{p\times q}N_1.
			\]
			Thus, $\phi$ can be expressed in the form given in 
			Theorem~\ref{thm: rank one gen}\,(2).
			
			\item When $n=1$, similarly, $\phi$ can be expressed in the form given in 
			Theorem~\ref{thm: rank one gen}\,(1).
			
		\end{itemize}
		
		The same argument applies when $\phi$ takes the form given in Theorem~\ref{thm: rank one gen}\,(4), 
		it can also be expressed in the form given in Theorem~\ref{thm: rank one gen}\,(1) or (2).
		This completes the proof. 
	\end{proof}

	\section{Rank-two nonincreasing and preserving maps}
	
	\subsection{The rank-two classification}
	
	Our main classification describes rank-two nonincreasing linear maps  
	$\M_{m,n}(\F) \to \M_{p,q}(\F)$ as follows.
	
	For every field, let
	\[
	\Alt_n(\F)=\{A\in\M_n(\F):A^\top=-A,\ a_{ii}=0\text{ for every }i\}
	\]
	be the space of alternating matrices. In characteristic different from two this is the skew-symmetric matrix space. The zero-diagonal condition is essential in characteristic two. Every nonzero matrix in $\Alt_3(\F)$ has rank two.
	
	\begin{theorem}\label{thm: M_mn rank 2 nonincreasing}
		Let $\F$ be a field with $|\F|\ge5$ and $\min\{m,n,p,q\}\ge2$. Let $\phi:  \M_{m,n}(\F) \to \M_{p,q}(\F)$ be a  rank-two nonincreasing linear map. Then $\phi$ has at least one of the following forms:
		\begin{enumerate}
			\item[(1)]  $\phi(A)=MAN$ for certain $M\in \M_{p,m}(\F)$ and $N\in \M_{n,q}(\F)$.  
			\item[(2)] $\phi(A)=MA^\top  N$ for certain $M\in \M_{p,n}(\F)$ and $N\in \M_{m,q}(\F)$.
			
			\item[(3)]  $\phi(A)=\delta (A) V^\top $ where $V\in \M_{q,2}(\F)$ is a constant matrix and $\delta:\M_{m,n}(\F)\to \M_{p,2}(\F)$ is  a linear map.
			\item[(4)]  $\phi(A)=U\eta(A)^\top $ where  $U\in \M_{p,2}(\F)$ is a constant matrix and $\eta:\M_{m,n}(\F)\to \M_{q,2}(\F)$ is a linear map.
			\item[(5)]  $\phi(A)=u\rho(A)^\top+\tau(A) v^\top =\mtx{u&\tau(A)} \mtx{\rho(A)&v}^\top $ for vectors $u\in \F^p$, $v\in \F^q$   and  linear maps $\rho:\M_{m,n}(\F)\to \F^q$ and  $\tau:\M_{m,n}(\F)\to \F^p$.  
			
			\item[(6)] 
			$\min\{p,q\}\ge 3$, and there are nonsingular matrices $M\in \GL_p(\F)$, $N\in \GL_q(\F)$, and a linear map $\gamma: \M_{m,n}(\F)\to \Alt_3(\F)$, such that 
			\begin{equation}\label{rank 2 so3}
				\phi(A)= M\mtx{\gamma(A) &0\\0 &0}_{p\times q} N.  
			\end{equation}
			
		\end{enumerate}
		Conversely, every map of these forms is rank-two nonincreasing.
	\end{theorem}
	
	The proof is given in Subsection~\ref{subsec:rank-two-proof}. 
	
	The first two forms are multiplication and transpose-multiplication maps. Forms (3) and (4) have a common row or column space of dimension at most two; form (5) is a sum of two rank-one terms; and form (6) has range equivalent to a subspace of an embedded copy of $\Alt_3(\F)$. 
	
	Forms (3)--(6) reflect the structure of  matrix spaces of upper rank 2 studied by Atkinson and Lloyd \cite{AtkinsonLloyd1981, Atkinson1983}. 
	The classification was extended by de Seguins Pazzis to arbitrary fields
	\cite{PAZZIS2012}. 
	
	\subsection{Rank-two preserving maps}
	
	Beasley and Laffey \cite{Beasley01} studied linear rank-two preservers on a fixed matrix space. The following consequence of Theorem~\ref{thm: M_mn rank 2 nonincreasing} describes rank-two preservers between matrix spaces of possibly different dimensions over fields with at least five elements.
	
	\begin{corollary}
		\label{thm: Mmn rank 2 preserver}
		Let $\F$ be a field with $|\F|\ge5$ and $\min\{m,n,p,q\}\ge2$. Then every  linear rank-two preserver $\phi:  \M_{m,n}(\F) \to \M_{p,q}(\F)$ takes at least one of the following forms:
		\begin{enumerate}
			\item[(1)]  $\phi(A)=MAN$ where $M\in \M_{p,m}(\F)$ has full column rank and $N\in \M_{n,q}(\F)$ has full row rank. 
			This is possible only when $p\ge m$ and $q\ge n$. 
			\item[(2)] $\phi(A)=MA^\top  N$ where $M\in \M_{p,n}(\F)$ has full column rank and  $N\in \M_{m,q}(\F)$ has full row rank. 
			This is possible only when $p\ge n$ and $q\ge m$. 
			
			\item[(3)]  $\phi(A)=\delta (A) V^\top $ where $V\in \M_{q,2}(\F)$ has rank 2  and $\delta:\M_{m,n}(\F)\to \M_{p,2}(\F)$ is  a linear map which is rank-two preserving, in other words, 
			$\{A\in \M_{m,n}(\F)\mid \rank (\delta(A))\le 1\}$ does not contain any matrix of rank 2.
			
			\item[(4)]  $\phi(A)=U\eta(A)^\top $ where  $U\in \M_{p,2}(\F)$ has rank 2  and $\eta:\M_{m,n}(\F)\to \M_{q,2}(\F)$ is a linear map which is rank-two preserving, in other words, 
			$\{A\in \M_{m,n}(\F)\mid \rank (\eta(A))\le 1\}$ does not contain any matrix of rank 2.
			\item[(5)]  $\phi(A)=u\rho(A)^\top+\tau(A) v^\top$ where $u\in \F^p\setminus\{0\}$, $v\in \F^q\setminus\{0\}$,   and  linear maps $\rho:\M_{m,n}(\F)\to \F^q$ and  $\tau:\M_{m,n}(\F)\to \F^p$ satisfy that 
			$\{A\in \M_{m,n}(\F)\mid \tau(A)\in\F u \text{ or } \rho(A)\in\F v\}$ does not contain any matrix of  rank $2$. 
			
			\item[(6)] 
			$\min\{p,q\}\ge3$, \((m,n)=(2,2),\ (3,2),\) or \((2,3)\), and there are nonsingular matrices \(M\in\GL_p(\F)\), \(N\in\GL_q(\F)\), and a linear map
			\(\gamma:\M_{m,n}(\F)\to\Alt_3(\F)\) such that 
			\[\phi(A)=M\mtx{\gamma(A)&0\\0&0}_{p\times q}N,
			\]
			in which one of the following cases holds.
			\begin{enumerate}
				\item[6a)] $(m,n)=(2,2),$ and $\ker\gamma$ is a matrix space of upper rank $1$. More precisely, $\ker\gamma$ has one of the following forms: 
				\begin{enumerate}
					\item [i.] $\ker\gamma =\F (uv^\top)$ for fixed nonzero vectors $u,v\in\F^2$;
					\item [ii.] $\ker\gamma =\{ uv^\top\mid u\in\F^2\}$ for a fixed  nonzero vector $v\in\F^2$;
					\item [iii.] $\ker\gamma =\{ uv^\top\mid v\in\F^2\}$ for a fixed  nonzero vector $u\in\F^2$. 
				\end{enumerate}
				
				\item[6b)] $(m,n)=(3,2)$ or $(2,3)$, and there is a nonsingular matrix $P\in\GL_3(\F)$ and a nonzero vector $v\in\F^2\setminus\{0\}$ such that 
				\[
				\gamma(A)=\mtx{0 &-p_ 1(A) &-p_ 2(A)\\p_ 1(A) &0 &-p_ 3(A)\\ p_ 2(A) &p_ 3(A) &0},  
				\]
				where  
				\[
				\mtx{p_ 1(A)\\p_ 2(A)\\p_ 3(A)} =\begin{cases}
					PAv, &\text{when $(m,n)=(3,2)$;} \\
					PA^\top v, &\text{when $(m,n)=(2,3)$.}
				\end{cases}
				\]
			\end{enumerate}

		\end{enumerate}
		Conversely, every map satisfying one of the listed sets of conditions preserves rank two.
	\end{corollary}

	\begin{remark} By \cite[Corollary 4]{PAZZIS2011}, every subspace  of $\M_{m,n}(\F)$ with codimension less than $\max\{m,n\}$ contains at least one rank $r$ matrix for every $r\in \{1, \ldots, \min\{m,n\}\}.$  
		\begin{enumerate}
			\item Consider the form $\phi(A)=\delta (A) V^\top $ in Corollary \ref{thm: Mmn rank 2 preserver} (3). Given a rank-one matrix $A_0\in\M_{p,2}(\F)$, we have 
			$$\delta^{-1}(\F A_0) \subseteq \{A\in \M_{m,n}(\F)\mid \rank (\delta(A))\le 1\}.$$
			Moreover,   the   linear map $\delta': \M_{m,n}(\F)\to \M_{p,2}(\F)/(\F A_0)$, given by
			$\delta'(A)=\delta(A)+\F A_0$, satisfies that $\ker \delta'=\delta^{-1}(\F A_0)$ and 
			\[
			\dim (\M_{m,n}(\F)) - \dim \ker \delta' = \dim \Ima\delta' \le \dim \M_{p,2}(\F) -\dim (\F A_0)=2p-1.
			\]
			Thus the codimension of $\delta^{-1}(\F A_0)$ in $\M_{m,n}(\F)$ is no more than $2p-1.$ 
			When
			$$2p-1<\max\{m,n\},$$
			\cite[Corollary 4]{PAZZIS2011} implies that the set  
			$\{A\in \M_{m,n}(\F)\mid \rank (\delta(A))\le 1\}$
			contains matrices of every rank up to $\min\{m,n\};$ thus 
			the form in Corollary \ref{thm: Mmn rank 2 preserver} (3) is impossible. 
			\item Analogously,  the form $\phi(A)=U\eta(A)^\top $ in Corollary \ref{thm: Mmn rank 2 preserver} (4) is impossible when
			$$2q-1<\max\{m,n\};$$
			
			\item The form $\phi(A)=u\rho(A)^\top+\tau(A) v^\top$ in Corollary \ref{thm: Mmn rank 2 preserver} (5) is impossible when
			$$
			\min\{p,q\}-1 <\max\{m,n\}.
			$$
		\end{enumerate}
		
	\end{remark}

	\subsection{Examples}\label{sec:example}
	
	\begin{example} \label{M-rank-two-preserver-examples}
		The following examples show that the forms in Corollary \ref{thm: Mmn rank 2 preserver} (3), (4), and (5) occur and can be distinct from the other forms. 
		Let $m=n=2$ and $p=q=4.$  
		
		\begin{enumerate}
			\item Let $\phi:\M_2(\F)\to \M_4(\F)$ such that  $\phi(A)=\delta (A) V^\top $, where $V\in \M_{4,2}(\F)$ has full column rank 2 and
			\begin{equation} \label{gen-rank-two-preserver-ex1}
				\delta: \M_2(\F)\to \M_{4,2}(\F),\qquad \delta\big(\mtx{a_{11} &a_{12}\\a_{21} &a_{22}} \big) =\mtx{a_{11} &a_{22}\\ a_{12} &a_{11}\\ a_{21} &a_{12}\\a_{22} &a_{21}}.
			\end{equation}
			The set  $$\{A\in \M_{2}(\F)\mid \rank (\delta(A))\le 1\}= \left\{a\mtx{1 &c\\c^2 &c^3}\mid a, c\in\F,\ c^4=1\right \}$$
			does not contain any matrix of rank 2. So $\phi$ preserves rank-two matrices and it takes the form of Corollary \ref{thm: Mmn rank 2 preserver} (3).
			
			The map $\phi$ sends the rank-one matrix $\mtx{1&0\\0&0}$ to a rank-two matrix; so $\phi$ cannot have  the form of Corollary \ref{thm: Mmn rank 2 preserver} (1) or (2). 
			Since the span of the columns  of $\phi(A)$ 
			over all $A\in \M_2(\F)$ is $\F^4$, $\phi$ cannot take the form of Corollary \ref{thm: Mmn rank 2 preserver} (4).
			Suppose, finally, that $\phi$ also has the form in  Corollary \ref{thm: Mmn rank 2 preserver} (5): 
			$$\phi(A)=\delta(A)V^\top =u\rho(A)^\top+\tau(A) v^\top$$ 
			where $u,v\in \F^4$   and   $\rho,\tau:\M_{2}(\F)\to \F^4$ are linear maps. Then there exists a vector $w\in\F^4$ such that
			$v^\top w=0$ but $V^\top w=\mtx{c_1\\c_2}\ne 0.$ For every $A=\mtx{a_{11} &a_{12}\\a_{21} &a_{22}}$ we get
			$$\phi(A)w=\delta(A) V^\top w=\mtx{c_1 a_{11} + c_2 a_{22}\\ c_1  a_{12} + c_2 a_{11}\\ c_1  a_{21} + c_2 a_{12}\\ c_1 a_{22} + c_2 a_{21}}
			=(u\rho(A)^\top+\tau(A) v^\top) w= (\rho(A)^\top w) u \in\F u
			$$
			This is impossible, so $\phi$ cannot have the form in  Corollary \ref{thm: Mmn rank 2 preserver} (5).
			
			\item Let $\phi':\M_2(\F)\to \M_4(\F)$ such that $\phi'(A)=\phi(A)^\top$ where $\phi$ is given above. Then  $\phi'$ preserves rank-two matrices and it takes the form of Corollary \ref{thm: Mmn rank 2 preserver} (4). Moreover, $\phi'$ cannot be expressed in forms (1), (2), (3), or (5). Both $\phi$ and $\phi'$ have four-dimensional images, so neither has form (6). 
			
			\item Let $\phi'':\M_2(\F)\to \M_4(\F)$  be given by
			\begin{equation} 
				\begin{aligned}
					\label{gen-rank-two-preserver-ex3}
					\phi''\big(\mtx{a_{11} &a_{12}\\a_{21} &a_{22}} \big) &= \mtx{a_{11}+a_{22} &a_{12} &a_{21} &a_{22}\\
						a_{11} &0&0&0\\ a_{12} &0&0&0\\ a_{21} &0&0&0}
					\\
					&=\mtx{1\\0\\0\\0}\mtx{a_{11} &a_{12} &a_{21} &a_{22}} +\mtx{a_{22}\\a_{11}\\a_{12} \\a_{21}}\mtx{1&0&0&0}.
				\end{aligned}
			\end{equation}
			We can check that $\phi''$ is a rank 2 preserver that takes the form of  Corollary \ref{thm: Mmn rank 2 preserver} (5) and $\phi''$ does not preserve any other rank. 
			
			Since $\rank \phi''\big(\mtx{0&1\\0&0}\big)=2$, $\phi''$ cannot have the form of  Corollary \ref{thm: Mmn rank 2 preserver} (1) or (2).
			If $\phi''$ can be expressed by  Corollary \ref{thm: Mmn rank 2 preserver} (3), namely $\phi''(A)=\delta (A) V^\top $, then the row space of $\phi''(A)$ for any $A\in \M_2(\F)$ is always contained in the 2-dimensional row space of $V^\top$. However, \eqref{gen-rank-two-preserver-ex3} shows that 
			the span of   rows  (resp.   columns) of $\phi''(A)$ over all $A\in \M_2(\F)$ is isomorphic to $\F^4$, which leads to a contradiction. Hence 
			$\phi''$ cannot take the form of  Corollary \ref{thm: Mmn rank 2 preserver} (3) (resp. Corollary \ref{thm: Mmn rank 2 preserver} (4)). 
			
		\end{enumerate}
	\end{example}
	
	\begin{example}
		A linear map that preserves a lower rank need not preserve a higher rank. 
		Consider a rank-two preserving linear map $\phi$ in  Corollary \ref{thm: Mmn rank 2 preserver} (3)--(6). The range of $\phi$ always has
		the upper rank  2. When $\min\{m,n\}\ge 3$,  $\phi$ is not a rank-$r$ preserver for any $r\in\{3,\ldots,\min\{m,n\}\}$. 
	\end{example}
	
	\begin{example}
		A linear map that preserves a higher rank need not preserve a lower rank either.  
		Let $p\ge 2 mn$  and consider a linear map  $\phi:  \M_{m,n}(\F) \to \M_{p,q}(\F)$ of the form $\phi(A)=\delta (A) V^\top $, in which $V\in \M_{q,2}(\F)$ has rank 2  and $\delta:\M_{m,n}(\F)\to \M_{p,2}(\F)$ has the form
		$$\delta(A) = \mtx{ \operatorname{vec}(A) &0\\ 0 & \operatorname{vec}(A)\\0 &0}$$
		where $ \operatorname{vec}(A)$ is the vectorization of $A$. Then $\phi$ is a rank-two preserver, but the image $\phi(A)$ for any nonzero $A$ always has rank 2 so that $\phi$ is not a rank-one preserver.
		
		This example also shows that nonincrease at a higher rank need not imply nonincrease at a lower rank. 
	\end{example}
	
	\subsection{Auxiliary lemmas}
	
	\begin{lemma}\label{thm: block matrix rank ineq}
		The following inequality holds for block matrices: 
		\begin{equation}
			\rank \mtx{A&B\\C&0}\ge \rank \mtx{0&B\\C&0}.
		\end{equation}
		Equality holds if and only if  $A=BX+ YC$ for certain  matrices $X$ and $Y$. 
	\end{lemma}
	
	\begin{proof}
		Let $r=\rank B$ and $s=\rank C$. By multiplying on the left and
		right by suitable invertible block diagonal matrices, we may assume
		\[
		B=\begin{pmatrix}I_r&0\\0&0\end{pmatrix},
		\qquad
		C=\begin{pmatrix}I_s&0\\0&0\end{pmatrix}.
		\]
		Partition $A$ conformally as
		$A=\left(\begin{smallmatrix}A_{11}&A_{12}\\A_{21}&A_{22}\end{smallmatrix}\right)$,
		where $A_{11}$ has size $r\times s$. Elementary row and column
		operations give
		\[
		\rank\begin{pmatrix}A&B\\ C&0\end{pmatrix}
		=r+s+\rank A_{22}\ge r+s= \rank\begin{pmatrix}0&B\\ C&0\end{pmatrix}. 
		\]
		Equality holds exactly when $A_{22}=0$. In these coordinates this is
		equivalent to $A=BX+YC$ for suitable $X$ and $Y$. Since all the
		reductions were made by invertible block diagonal matrices, the same
		criterion holds in the original coordinates.
	\end{proof}
	
	We will use the following special case of Lemma \ref{thm: block matrix rank ineq}.
	
	\begin{corollary}\label{thm: rank ineq rank 2}
		Suppose matrices $B\ne 0$ and $C\ne 0$, and $\rank \mtx{A&B\\C&0}\le 2$. Then
		there exist nonzero vectors $u_1, u_2, v_1, v_2$ and vectors $p, q$   such that
		\begin{equation}
			B=u_1v_2^\top ,\qquad C=u_2v_1^\top ,\qquad A=u_1p^\top +qv_1^\top .
		\end{equation}
	\end{corollary}
	
	\begin{proof}
		Since
		\[
		\rank\mtx{0&B\\C&0}=\rank B+\rank C,
		\]
		Lemma~\ref{thm: block matrix rank ineq} and the assumptions $B,C\ne0$
		give $\rank B=\rank C=1$ and equality in the block rank inequality.
		Write $B=u_1v_2^\top$ and $C=u_2v_1^\top$. The equality criterion gives
		$A=BX+YC=u_1p^\top+qv_1^\top$, where $p=X^\top v_2$ and $q=Yu_2$.
	\end{proof}
	
	\begin{lemma}\label{thm: rank one outside proper subspaces}
		Let $\F$ be a field with $|\F|>k$. Let $S_1,\ldots,S_k$ be proper subspaces of $\M_{m,n}(\F)$.
		Then there is a rank-one matrix $A\in \M_{m,n}(\F)$ such that $A\not\in \bigcup_{i=1}^{k} S_i$. 
	\end{lemma}
	
	\begin{proof}
		For each $i$, choose a nonzero linear functional $\lambda_i$ on
		$\M_{m,n}(\F)$ that vanishes on $S_i$. The polynomial
		\[
		P_i(u,v)=\lambda_i(uv^\top),
		\qquad u\in\F^m,\quad v\in\F^n,
		\]
		is nonzero because rank-one matrices span $\M_{m,n}(\F)$. Hence
		$P=\prod_{i=1}^kP_i$ is a nonzero polynomial of degree at most $k$ in each scalar entry of $u$ and $v$. Since $|\F|>k$, the polynomial fact from Section~\ref{sec:rank-relations} shows that $P$ cannot vanish at every point of $\F^m\times\F^n$. Choose $(u,v)$ with $P(u,v)\ne0$. Then $u,v\ne0$
		and $uv^\top\notin S_i$ for every $i$.
	\end{proof}
	
	\begin{corollary}\label{thm: rank one basis outside}
		Let $\F$ be a field with $|\F|>k+1$. Let $S_1,\ldots,S_k$ be proper subspaces of $\M_{m,n}(\F)$.
		There is a basis $ \{A_{i}\mid i\in [mn]\}$ of  $\M_{m,n}(\F)$ such that each $A_{i}$ is a rank-one matrix and
		\begin{equation}\label{rank one basis outside subspaces}
			\{A_{i}\mid i\in [mn]\}\cap \left(\bigcup_{i=1}^{k} S_i\right)=\emptyset.
		\end{equation}
	\end{corollary}
	
	\begin{proof}
		Choose the matrices successively. Having chosen linearly independent
		rank-one matrices $A_1,\ldots,A_{j-1}$, apply
		Lemma~\ref{thm: rank one outside proper subspaces} to
		\[
		S_1,\ldots,S_k,\quad \span\{A_1,\ldots,A_{j-1}\}.
		\]
		The additional subspace is proper when $j\le mn$, and the inequality $|\F|>k+1$ permits this application of the lemma. The resulting matrices
		$A_1,\ldots,A_{mn}$ form the required basis.
	\end{proof}
	
	\begin{lemma}\label{lem:cubic-obstruction}
		Let $u'\in\F^{p-2}$ and $v'\in\F^{q-2}$ be nonzero, and let
		$G:\M_{m,n}(\F)\to\M_2(\F)$ and
		$\delta,\eta:\M_{m,n}(\F)\to\F^2$ be linear. Suppose that
		\[
		\Phi(X)=\begin{pmatrix}
			G(X)&\delta(X){v'}^\top\\
			u'\eta(X)^\top&0
		\end{pmatrix}
		\]
		has rank at most two whenever $X$ has rank at most two. Then
		\begin{equation}\label{eq:cubic-identity}
			f(X):=\det\begin{pmatrix}G(X)&\delta(X)\\\eta(X)^\top&0\end{pmatrix}
			=\delta(X)^\top JG(X)J\eta(X)=0
			\qquad(X\in\M_{m,n}(\F)),
		\end{equation}
		where $J=\begin{pmatrix}0&-1\\1&0\end{pmatrix}$.
	\end{lemma}
	
	\begin{proof}
		Choose nonzero entries $u'_i$ and $v'_j$. When $\rank X\le2$, the scalar $u'_iv'_j f(X)$ is a $3\times3$ minor of $\Phi(X)$, so $f(X)=0$. Let $\K$ be an algebraic closure of $\F$. The polynomial $f$ is homogeneous of degree three, unless it is zero. Since $|\F|\ge5$, the factorization argument in Section~\ref{sec:rank-relations}, with $s=2$ and $d=3$, shows that $f$ vanishes on all matrices of rank at most two over $\K$ as well. If $\min\{m,n\}\le2$, this gives $f=0$ as a polynomial.
		
		Assume $m,n\ge3$. Over $\K$, the ideal of the rank-at-most-two determinantal variety is generated by the $3\times3$ minors (see for example \cite[Theorem 2.10]{BrunsVetter1988}). Its degree-three part is therefore their $\K$-linear span. Hence
		\[
		f(X)=\sum_{\substack{|I|=3\\|L|=3}}c_{I,L}\det X[I;L],
		\qquad c_{I,L}\in \K.
		\]
		On the other hand, writing $G=(g_{ab})$, $\delta=(\delta_1,\delta_2)^\top$, and $\eta=(\eta_1,\eta_2)^\top$, we have
		\[
		f=-\delta_1g_{22}\eta_1+\delta_1g_{21}\eta_2
		+\delta_2g_{12}\eta_1-\delta_2g_{11}\eta_2.
		\]
		Thus $f$ is a sum of at most four products of three linear forms. If $c_{I,L}\ne0$, set every entry outside $X[I;L]$ equal to zero. Every other $3\times3$ minor then vanishes, so, after division by $c_{I,L}$, the generic $3\times3$ determinant would be a sum of at most four products of linear forms over $\K$. This contradicts its Chow-rank lower bound of five, valid in every characteristic (cf.   \cite{IltenTeitler2016, IltenSuess2020}). Thus every $c_{I,L}$ is zero and $f=0$.
		
		For completeness, the exact Chow rank is five over every field. The lower bound just used combines with the characteristic-free five-term decomposition in \cite[Section~2]{KrishnaMakam2021}. Equivalently, for nine independent variables one may verify directly that
		\[
		\begin{aligned}
			\det\begin{pmatrix}a&b&c\\d&e&f\\g&h&i\end{pmatrix}
			={}&(b+c)dh-(a+c)eg-b(d+f)(h+i)\\
			&+(b-a)f(g+h+i)+a(e+f)(g+i).
		\end{aligned}
		\]
		The identity has integral coefficients, so it is valid in characteristic two as well. Only the Chow-rank lower bound is needed for the obstruction above.
	\end{proof}
	
	\begin{lemma}\label{lem:binary-normal-form}
		Let $G:\M_{m,n}(\F)\to\M_2(\F)$ and
		$\delta,\eta:\M_{m,n}(\F)\to\F^2$ be linear, with the two coordinates of each of $\delta$ and $\eta$ linearly independent. Suppose
		\[
		\delta(X)^\top JG(X)J\eta(X)=0
		\qquad(X\in\M_{m,n}(\F)),
		\quad J=\begin{pmatrix}0&-1\\1&0\end{pmatrix}.
		\]
		If $\span\{\delta_1,\delta_2\}\ne\span\{\eta_1,\eta_2\}$, there are constant vectors $a,b\in\F^2$ such that
		\begin{equation}\label{eq:normal-form-distinct}
			G(X)=\delta(X)a^\top+b\eta(X)^\top.
		\end{equation}
		If the two spans agree, write $\eta=S\delta$ for $S\in\GL_2(\F)$. Then there are constant vectors $a,b\in\F^2$ and a linear functional $\lambda$ such that
		\begin{equation}\label{eq:normal-form-equal}
			G(X)=\delta(X)a^\top+b\eta(X)^\top
			+\lambda(X)S^{-1}J.
		\end{equation}
	\end{lemma}
	
	\begin{proof}
		The assumed pointwise cubic identity is a polynomial identity because $|\F|>3$, by the polynomial fact in Section~\ref{sec:rank-relations}. Regard the coordinates of $\delta$, $\eta$, and $G$ as linear forms in $\F[x_{ij}]$, and set
		$R=GJ\eta=(R_1,R_2)^\top$. The assumed identity gives
		$\delta_2R_1-\delta_1R_2=0$. Since $\delta_1$ and $\delta_2$ are relatively prime, there is a linear form $\ell$ with
		\begin{equation}\label{eq:R-delta-ell}
			R=\delta\ell.
		\end{equation}
		Set $\mathfrak a=(\eta_1,\eta_2)$. Both components of $R$ lie in $\mathfrak a$, so $\delta_1\ell,\delta_2\ell\in\mathfrak a$. The quotient $\F[x_{ij}]/\mathfrak a$ is an integral domain.
		
		If the coordinate spans differ and $\ell\notin\mathfrak a$, then $\delta_1,\delta_2\in\mathfrak a$, forcing the two spans to agree. Thus $\ell\in\mathfrak a$, and for a constant vector $a$ we may write $\ell=a^\top J\eta$. By \eqref{eq:R-delta-ell}, the linear map $T=G-\delta a^\top$ satisfies $TJ\eta=0$. Each row $(t_1,t_2)$ of $T$ therefore satisfies $-t_1\eta_2+t_2\eta_1=0$. Since $\eta_1,\eta_2$ are relatively prime linear forms, each row is a constant multiple of $\eta^\top$. Hence $T=b\eta^\top$, proving \eqref{eq:normal-form-distinct}.
		
		Now suppose $\eta=S\delta$. If $Z\in\ker\delta$, apply the cubic identity to $X+tZ$ and compare the coefficient of $t$. Since $\delta$ is onto $\F^2$, this gives
		\[
		d^\top JG(Z)JSd=0\qquad(d\in\F^2).
		\]
		Thus $JG(Z)JS$ is alternating, including in characteristic two, so $G(Z)=\lambda_0(Z)S^{-1}J$ for a scalar $\lambda_0(Z)$. Extend the resulting linear functional $\lambda_0$ on $\ker\delta$ to a linear functional $\lambda$ on $\M_{m,n}(\F)$. Then $T=G-\lambda S^{-1}J$ vanishes on $\ker\delta$ and factors as $T(X)=L(\delta(X))$ for a linear map $L:\F^2\to\M_2(\F)$. The term $S^{-1}J$ itself satisfies the cubic identity, so
		\[
		d^\top JL(d)JSd=0\qquad(d\in\F^2).
		\]
		To make the final coefficient comparison explicit, put $\widehat L(d)=L(d)S^{-\top}$. The identity $JS=(\det S)S^{-\top}J$ gives
		\[
		d^\top J\widehat L(d)Jd=0.
		\]
		Write $\widehat L\binom{x}{y}=xA+yB$. Comparing the four cubic coefficients yields
		\[
		A_{22}=B_{11}=0,\qquad
		B_{22}=A_{12}+A_{21},\qquad
		A_{11}=B_{12}+B_{21}.
		\]
		Consequently $\widehat L(d)=da_0^\top+b_0d^\top$, where $a_0=(B_{21},A_{12})^\top$ and $b_0=(B_{12},A_{21})^\top$. Multiplication by $S^\top$ gives $L(d)=da^\top+b(Sd)^\top$, with $a=Sa_0$ and $b=b_0$. This proves \eqref{eq:normal-form-equal} without division by any integer.
	\end{proof}
	
	\subsection{Proof of the classification of rank-two nonincreasing maps}\label{subsec:rank-two-proof}
	
	Given $A\in\M_{m,n}(\F)$ and index sets $J\subseteq[m]$ and $K\subseteq [n],$  let $A[J;K]$ denote the submatrix with rows and columns indexed by $J$ and $K$, respectively.  Let $\col{A}$ (resp. $\row(A)$) denote the column space (resp. row space) of a matrix $A$. 
	
	We now prove the classification theorem. 
	
	\begin{proof}[Proof of Theorem \ref{thm: M_mn rank 2 nonincreasing}]
		The converse follows directly from the six displayed forms: forms (1) and (2) do not increase any rank, and forms (3)--(6) have range of upper rank at most two. We prove necessity by normalization, restrictions on the off-diagonal blocks, and a final compatibility analysis. The last part has four cases according to the fixed row or column factors of those blocks.
		
		\medskip\noindent\emph{Step 1: normalization.}
		We begin with the images of rank-one matrices. If $\rank(A)\le 1$ implies that $\rank(\phi(A))\le 1$ for  $A\in \M_{m,n}(\F)$, then
		$\phi$ takes one of the forms in 
		Theorem \ref{thm: rank decreasing one gen} (1)--(4), each of which is
		contained in forms (1)--(4) of the present theorem.
		
		Now assume that there is a rank-one matrix $A_0\in \M_{m,n}(\F) $ such that $\rank(\phi(A_0))=2$. There are nonsingular matrices $P\in \M_p(\F)$ and $Q\in \M_q(\F)$ such that $\phi(A_0)=P\mtx{I_2 &0\\0&0} Q$. The map $\wt{\phi}: \M_{m,n}(\F)\to \M_{p,q}(\F)$ given by $\wt{\phi}(C)=P^{-1}\phi(C)Q^{-1}$ is rank-two nonincreasing. 
		
		Suppose there is a rank-one matrix $B\in \M_{m,n}(\F)$ such that  $\wt{\phi}(B)=\mtx{B_{11}&B_{12}\\B_{21} &B_{22}}$ and $B_{22}\ne  0$, say
		the $(i,j)$ entry of $\wt{\phi}(B)$ is nonzero
		for some $2<i\le p$ and $2<j\le q$. 
		The $3\times3$ minor   of $$\wt{\phi}(xA_0+B) =x\wt{\phi}(A_0) +\wt{\phi}(B)=\mtx{xI_2+B_{11}&B_{12}\\B_{21} &B_{22}}$$ indexed by rows $1, 2,  i$ and columns $1, 2, j$ 
		is a quadratic polynomial of $x$.
		
		The quadratic has a nonzero leading coefficient and hence at most two roots. Since $|\F|\ge5$, we can choose a nonzero scalar $c\in\F$ such that 
		the above minor is nonzero for $x=c$, so that 
		$\rank (\wt{\phi}(cA_0+B)) \ge 3.$   However, $\rank (cA_0+B)\le \rank(A_0)+\rank(B)=2$, contradicting the hypothesis on $\wt{\phi}$. Hence $B_{22}=0$. 
		
		Since rank-one matrices span $\M_{m,n}(\F)$, the map $\wt{\phi}$ has the form:
		$$\wt{\phi}(B)=\mtx{\wt{\phi}_{11}(B) &\wt{\phi}_{12}(B)\\\wt{\phi}_{21}(B) &0}$$
		where each $\wt{\phi}_{ij}$ is a linear map. 
		
		We consider two special cases first:
		\begin{enumerate}
			\item  Suppose  $\wt{\phi}_{21}= 0$. Then $\wt{\phi}(B)=\mtx{\wt{\phi}_{11}(B) &\wt{\phi}_{12}(B)\\0 &0}$. 
			So   
			$$\wt{g} :\M_{m,n}(\F)\to \M_{q,2}(\F),\qquad \wt{g}(B)=\mtx{\wt{\phi}_{11}(B) &\wt{\phi}_{12}(B)}^\top  $$ is a linear map and
			$$        \wt{\phi}(B)  =  P^{-1}\phi(B)Q^{-1} =\mtx{I_2\\0} \wt{g}(B)^\top,\qquad 
			\phi(B)  =  \left(P\mtx{I_2\\0} \right) \left(Q^\top \wt{g}(B) \right)^\top  . $$
			Hence $\phi$ has the form in Theorem \ref{thm: M_mn rank 2 nonincreasing} (4).
			
			\item   Suppose  $\wt{\phi}_{12}= 0$. Similar arguments show that $\phi$ should take the form in Theorem \ref{thm: M_mn rank 2 nonincreasing} (3).
		\end{enumerate}
		
		\medskip\noindent\emph{Step 2: the off-diagonal maps.}
		From now on, assume that $\wt{\phi}_{21}\ne 0$ and $\wt{\phi}_{12}\ne 0$. 
		
		We claim that both $\wt{\phi}_{12}: \M_{m,n}(\F)\to \M_{2,q-2}(\F)$ and $\wt{\phi}_{21}: \M_{m,n}(\F)\to \M_{p-2,2}(\F)$ are rank-one nonincreasing maps. 
		Otherwise, without loss of generality, let
		$B'\in \M_{m,n}(\F)$ be a rank-one matrix such that  $\rank \wt{\phi}_{12}(B')\ge 2$;
		then $$\det(\wt{\phi}_{12}(B')[J;L])\ne 0$$ for a $2$-minor of $\wt{\phi}_{12}(B')$ indexed by $J$ and $L$. 
		Let $C$ be a rank-one matrix such that $\wt{\phi}_{21}(C)\ne 0$. We use the following observations.
		\begin{enumerate}
			\item  
			There is at most one value $x\in\F$ such that $$ \wt{\phi}_{21}(C+xB')= \wt{\phi}_{21}(C)+x\wt{\phi}_{21}(B')=0.$$
			\item 
			The minor $\det(\wt{\phi}_{12}(C+xB')[J;L])$ is a degree 2 polynomial of $x$; there are at most two values $x\in\F$ such that 
			$$\det(\wt{\phi}_{12}(C+xB')[J;L])=0.$$ 
			At most three scalar values are excluded in total. Since $|\F|\ge5$, we can choose a scalar $c\in\F$ such that
			\[
			\wt{\phi}_{21}(C+cB')\ne 0,\qquad
			\rank(\wt{\phi}_{12}(C+cB'))\ge 2.
			\]
			Now  $\rank (C+cB')\le \rank C+\rank B'=2$ but
			$$
			\rank \wt{\phi} (C+cB')\ge \rank \wt{\phi}_{21}(C+cB')+\rank \wt{\phi}_{12}(C+cB')\ge 3,
			$$
			contradicting the assumption that   $\wt{\phi}$  is rank-two nonincreasing. 
		\end{enumerate}
		
		Since $\wt{\phi}_{12}$ and $\wt{\phi}_{21}$ are rank-one nonincreasing, each has one of the four forms
		given in Theorem \ref{thm: rank decreasing one gen}. We claim that neither off-diagonal map can have an image containing a rank-two matrix. By Botta's theorem, it suffices to exclude multiplication and transpose-multiplication forms with such an image. 
		
		Suppose first, to the contrary, that $\wt{\phi}_{12}(X)=MXN$  where
		$M\in\M_{2,m}(\F)$ has full row rank $2$ and $N\in\M_{n,q-2}(\F)$ has $\rank N\ge 2.$  
		
		The set $\ker\wt{\phi}_{12}\cup \ker\wt{\phi}_{21}$ is a union of two proper subspaces of $\M_{m,n}(\F)$.
		Since $|\F|>2$, Lemma \ref{thm: rank one outside proper subspaces} gives a rank-one matrix 
		$A_1 \in \M_{m,n}(\F)$ not in $\ker\wt{\phi}_{12}\cup \ker\wt{\phi}_{21}.$ 
		By the rank-one nonincreasing property  of $\wt{\phi}_{12}$ and $\wt{\phi}_{21}$, 
		$$\rank \wt{\phi}_{12}(A_{1})=\rank \wt{\phi}_{21}(A_{1})=1.$$  
		There exists $X\in \M_{m,n}(\F)$ such that  $\wt{\phi}_{12}(X)=MXN$ has rank $2$. Hence 
		the sets
		\[
		\begin{aligned}
			\mathcal S_{\mathrm{col}} &=\{B\in \M_{m,n}(\F)\mid \col(\wt{\phi}_{12}(B))\subseteq\col(\wt{\phi}_{12}(A_1))\} \quad \text{and}
			\\
			\mathcal S_{\mathrm{row}} &= \{B\in \M_{m,n}(\F)\mid \row(\wt{\phi}_{12}(B))\subseteq\row(\wt{\phi}_{12}(A_1))\}
		\end{aligned}
		\] 
		are proper subspaces of $\M_{m,n}(\F).$ Since $|\F|>4$, we may apply Lemma \ref{thm: rank one outside proper subspaces} to the following four subspaces. Thus
		there is a rank-one matrix 
		$A_2 \in \M_{m,n}(\F)$ that is not in
		\[
		\ker\wt{\phi}_{12}\cup \ker\wt{\phi}_{21}\cup\mathcal S_{\mathrm{col}}\cup\mathcal S_{\mathrm{row}}.
		\]
		Write $A_1=u_1v_1^\top$ and $A_2=u_2v_2^\top$ for $u_1,u_2\in\F^{m}$ and $v_1,v_2\in\F^n.$
		Then 
		\[
		\wt{\phi}_{12}(A_{1})=(Mu_1)(N^\top v_1)^\top,\qquad \wt{\phi}_{12}(A_{2})=(Mu_2)(N^\top v_2)^\top,
		\]
		where $Mu_1$ and $Mu_2$ (resp. $N^\top v_1$ and $N^\top v_2$) are not collinear. 
		Choose $c\in\F\setminus\{0\}$ such that 
		\[
		\wt{\phi}_{21}(A_1+cA_2)=\wt{\phi}_{21}(A_1)+c\wt{\phi}_{21}(A_2)\ne 0.
		\]
		We have  $\rank(A_1+cA_2)=2.$ However, $\rank \wt{\phi}_{21}(A_1+cA_2)\ge 1$ and
		\[
		\rank \wt{\phi}_{12}(A_1+cA_2)=\rank((Mu_1)(N^\top v_1)^\top+c(Mu_2)(N^\top v_2)^\top) =2,  
		\]
		so that
		\[
		\rank \wt{\phi}(A_1+cA_2)\ge \rank \wt{\phi}_{12}(A_1+cA_2)+\rank \wt{\phi}_{21}(A_1+cA_2)\ge 3,
		\]
		which contradicts the assumption that   $\wt{\phi}$  is rank-two nonincreasing.         
		Thus $\wt{\phi}_{12}$ cannot have a multiplication form with rank-two image. 
		
		The same argument applies to $\wt{\phi}_{21}$ and, after transposition, to either off-diagonal map in multiplication form (2) of Theorem~\ref{thm: rank decreasing one gen}. 
		
		Therefore, each off-diagonal map has an image of upper rank at most one and can be written in form (3) or (4) of Theorem~\ref{thm: rank decreasing one gen}. 
		
		\medskip\noindent\emph{Step 3: compatibility of the off-diagonal forms.}
		If the two blocks have complementary fixed factors, we obtain form (5). Matching fixed columns or rows give forms (4) and (3), respectively. The remaining crossed orientation leads to forms (5) and (6).
		\begin{enumerate}
			\item  [(a)]  
			Suppose that
			$$\wt{\phi}_{12}(X)=u\eta_{12}(X)^\top ,\qquad \wt{\phi}_{21}(X)=\delta_{21}(X) v^\top ,$$
			for  nonzero  constant vectors $u, v\in\F^2$ and    linear maps $\eta_{12}:\M_{m,n}(\F)\to \F^{q-2}$  and 
			$\delta_{21}:\M_{m,n}(\F)\to \F^{p-2}$. 
			Since $|\F|>3$, Corollary \ref{thm: rank one basis outside}, applied to the two kernels, gives a basis $\{X_i\mid i\in [mn]\}$ of $\M_{m,n}(\F)$ consisting of rank-one matrices  such that 
			$\eta_{12}(X_i)\ne 0$ and $\delta_{21}(X_i)\ne 0$ for all $i\in [mn]$.  
			By Corollary \ref{thm: rank ineq rank 2}, there exist vectors $a_i, b_i\in\F^2$ such that
			\begin{equation}
				\wt{\phi}_{11}(X_i)=ua_i^\top +b_iv^\top ,\qquad i\in [mn]. 
			\end{equation}
			We have for $i\in [mn]$: 
			\begin{equation}
				\wt{\phi}(X_i)=\mtx{ua_i^\top +b_iv^\top  &u\eta_{12}(X_i)^\top  \\ \delta_{21}(X_i) v^\top  &0} = \mtx{b_i\\\delta_{21}(X_i)}\mtx{v\\0}^\top +\mtx{u\\0}\mtx{a_i\\\eta_{12}(X_i)}^\top  .
			\end{equation}
			Let $\wt{u}= \mtx{u\\0}$, $\wt{v}=\mtx{v\\0}$. Define the  linear maps  $\wt{g}:\M_{m,n}(\F)\to \F^{q}$ 
			and    $\wt{f}:\M_{m,n}(\F)\to \F^{p}$   such that for $i\in [mn]$:
			$$\wt{f}(X_i)= \mtx{b_i\\\delta_{21}(X_i)},\qquad \wt{g}(X_i)=\mtx{a_i\\\eta_{12}(X_i)}.$$
			Then $\wt{\phi}(X)=\wt{f}(X)\wt{v}^\top + \wt{u}\wt{g}(X)^\top  $  has the form in Theorem \ref{thm: M_mn rank 2 nonincreasing} (5). The same holds for
			$\phi$. 
			
			\item  [(b)]  
			Suppose that
			$$\wt{\phi}_{12}(X)=u\eta_{12}(X)^\top ,\qquad \wt{\phi}_{21}(X)=u' \eta_{21}(X)^\top , $$
			for nonzero  constant vectors $u\in\F^2$ and  $u'\in\F^{p-2}$  and nonzero linear maps $\eta_{12}:\M_{m,n}(\F)\to \F^{q-2}$  and   $\eta_{21}:\M_{m,n}(\F)\to \F^{2}$. If $\dim (\Ima \eta_{21})\le 1$ then we revert to  case (a). Now assume that  $\dim (\Ima \eta_{21})\ge 2$. 
			Since $|\F|>3$, Corollary \ref{thm: rank one basis outside}, applied to the two kernels, gives a basis $\{X_i\mid i\in [mn]\}$ of $\M_{m,n}(\F)$ consisting of rank-one matrices,  such that 
			$\eta_{12}(X_i)\ne 0$ and $\eta_{21}(X_i)\ne 0$ for all $i\in [mn]$.  
			By Corollary \ref{thm: rank ineq rank 2}, there exist vectors $a_i, b_i\in\F^2$ such that
			$\wt{\phi}_{11}(X_i)=ua_i^\top +b_i\eta_{21}(X_i)^\top$, so that
			\begin{equation}\label{wt-phi-Xi-1}
				\wt{\phi}(X_i)=\mtx{ua_i^\top +b_i\eta_{21}(X_i)^\top  &u\eta_{12}(X_i)^\top  \\ u' \eta_{21}(X_i)^\top  &0},\qquad
				i\in [mn].
			\end{equation}
			By assumption, $\wt{\phi}_{21}(X_i)\ne 0$ so that $\eta_{21}(X_i)\ne 0$ for all $i\in [mn]$.  
			Since $\dim (\Ima \eta_{21})\ge 2$, there exists $j\in [mn]$ such that
			$\eta_{21}(X_1)$ and $\eta_{21}(X_j)$ are noncollinear. 
			Since $\eta_{12}(X_i)\ne 0$ and $\eta_{21}(X_i)\ne 0$ for $i\in [mn]$, 
			we can choose $y\in\F\setminus\{0\}$ such that $\eta_{12}(X_1+yX_j)\ne 0$ and $\eta_{21}(X_1+yX_j)\ne 0$. 
			Only $y=0$ and at most two additional values are excluded, so this choice is possible under $|\F|\ge5$. Then $\rank (X_1+yX_j)\le 2$ so that $\rank\wt{\phi}(X_1+yX_j)\le 2$. By \eqref{wt-phi-Xi-1} and Corollary \ref{thm: rank ineq rank 2}, there exist vectors $a',b'\in\F^2$ such that
			\begin{eqnarray}
				&& ua_1^\top  +yua_j^\top +b_1\eta_{21}(X_1)^\top +yb_j\eta_{21}(X_j)^\top 
				\notag \\ &=& \wt{\phi}_{11}(X_1+yX_j) 
				= u{a'}^\top +b' \eta_{21}(X_1+yX_j)^\top .
			\end{eqnarray}
			Therefore,
			\begin{eqnarray}
				u(a'-a_1-ya_j)^\top &=& (b_1-b')\eta_{21}(X_1)^\top +y(b_j-b')\eta_{21}(X_j)^\top 
				\notag\\
				&=& \mtx{b_1-b' & b_j-b'} \mtx{\eta_{21}(X_1) & y\eta_{21}(X_j)}^\top.
			\end{eqnarray}
			The rank of the above matrix is no more than $1$. Since $\eta_{21}(X_1)$ and $y\eta_{21}(X_j)$ are noncollinear, both $b_1-b'$ and $b_j-b'$ are collinear to $u$.
			So $b_j-b_1=(b_j-b')-(b_1-b')=c_j u$ for some $c_j\in\F.$ 
			
			For any $k\in [mn]$, $\eta_{21}(X_k)$ is noncollinear to at least one of $\eta_{21}(X_1)$ and $\eta_{21}(X_j)$. 
			If $\eta_{21}(X_k)$ is noncollinear to  $\eta_{21}(X_1)$,  then similar arguments show  that $b_k-b_1=c_k u$ for some $c_k\in\F$.
			Otherwise, $\eta_{21}(X_k)$ is noncollinear to  $\eta_{21}(X_j)$, then $b_k-b_j$ is collinear to $u$, so that
			$b_k-b_1=(b_k-b_j)+(b_j-b_1)=c_k u$ for some $c_k\in\F$. Overall, write $b=b_1$, then 
			there exist $c_1=0,c_2,\ldots,c_{mn}\in\F$ such that
			\begin{equation}
				b_k=b+c_ku,\qquad k\in [mn].
			\end{equation}
			\eqref{wt-phi-Xi-1} becomes
			\begin{eqnarray}
				\wt{\phi}(X_k)&=&\mtx{u a_k^\top +(b+c_k u)\eta_{21}(X_k)^\top  &u\eta_{12}(X_k)^\top  \\ u' \eta_{21}(X_k)^\top  &0}
				\notag \\
				&=& \mtx{u &b\\0 &u'}\mtx{ a_k+ c_k \eta_{21}(X_k) & \eta_{21}(X_k)\\ \eta_{12}(X_k) & 0}^\top .
			\end{eqnarray} 
			Let $\wt{U}=\mtx{u &b\\0 &u'}\in \M_{p,2}(\F)$ and define a  linear map $\wt{g}:\M_{m,n}(\F)\to \M_{q,2}(\F)$  by 
			$$\wt{g}(X_k)=\mtx{ a_k+ c_k \eta_{21}(X_k) & \eta_{21}(X_k)\\ \eta_{12}(X_k) & 0},\qquad k\in [mn].$$
			Then $\wt{\phi}(X)= \wt{U} \wt{g}(X)^\top $ has the form in Theorem \ref{thm: M_mn rank 2 nonincreasing} (4). The same holds for
			$\phi$. 
			
			\item [(c)] Suppose that
			$$\wt{\phi}_{12}(X)= \delta_{12}(X) {v'}^\top ,\qquad \wt{\phi}_{21}(X)=\delta_{21}(X) v^\top ,$$
			for  nonzero  constant vectors $v'\in\F^{q-2}$ and $v\in\F^{2}$  and linear maps 
			$\delta_{12}:\M_{m,n}(\F)\to \F^{2}$
			and $\delta_{21}:\M_{m,n}(\F)\to \F^{p-2}$. Then the linear map $X\mapsto \wt{\phi}(X)^\top $ meets the criteria of case (b), so that it has the form in Theorem \ref{thm: M_mn rank 2 nonincreasing} (4). Hence both $\wt{\phi}$ and $\phi$ have the form in Theorem \ref{thm: M_mn rank 2 nonincreasing} (3).
			
			\item[(d)] Suppose that
			\[
			\wt{\phi}_{12}(X)=\delta(X){v'}^\top,
			\qquad
			\wt{\phi}_{21}(X)=u'\eta(X)^\top,
			\]
			where $v'\in\F^{q-2}$ and $u'\in\F^{p-2}$ are nonzero, and $\delta,\eta:\M_{m,n}(\F)\to\F^2$ are linear. If either coordinate pair is dependent, this case reduces to (a), (b), or (c). We therefore assume both pairs are independent. Put $G=\wt{\phi}_{11}$ and $J=\begin{pmatrix}0&-1\\1&0\end{pmatrix}$. Lemma~\ref{lem:cubic-obstruction} gives
			\[
			\delta(X)^\top JG(X)J\eta(X)=0
			\qquad(X\in\M_{m,n}(\F)).
			\]
			Lemma~\ref{lem:binary-normal-form} now distinguishes two possibilities.
			\begin{enumerate}
				\item[(i)] If $\span\{\delta_1,\delta_2\}\ne\span\{\eta_1,\eta_2\}$, then $G=\delta a^\top+b\eta^\top$ for constant vectors $a,b\in\F^2$. Hence
				\[
				\wt{\phi}(X)
				=\mtx{\delta(X)\\0}\mtx{a\\v'}^\top
				+\mtx{b\\u'}\mtx{\eta(X)\\0}^\top,
				\]
				which is form (5).
				
				\item[(ii)] If the spans agree, write $\eta=S\delta$ with $S\in\GL_2(\F)$. Lemma~\ref{lem:binary-normal-form} gives
				\[
				G(X)=\delta(X)a^\top+b\eta(X)^\top
				+\lambda(X)S^{-1}J
				\]
				for constant vectors $a,b$ and a linear functional $\lambda$. Multiplying $\wt{\phi}(X)$ on the right by $\diag(S^{-\top},I_{q-2})$, and relabeling $S^{-1}a$ as $a$ and the nonzero scalar multiple of $\lambda$ as $\lambda$, gives
				\[
				\Phi(X)=
				\begin{pmatrix}
					\delta(X)a^\top+b\delta(X)^\top+\lambda(X)J
					&\delta(X){v'}^\top\\
					u'\delta(X)^\top&0
				\end{pmatrix}.
				\]
				Choose $L\in\M_{2,p-2}(\F)$ and $R\in\M_{2,q-2}(\F)$ with $Lu'=b$ and $Rv'=a$. Choose $T_p\in\GL_{p-2}(\F)$ and $T_q\in\GL_{q-2}(\F)$ such that $T_pu'=e_1^{(p-2)}$ and $T_qv'=-e_1^{(q-2)}$. Then
				\[
				\begin{pmatrix}I_2&-L\\0&T_p\end{pmatrix}
				\Phi(X)
				\begin{pmatrix}I_2&0\\-R^\top&T_q^\top\end{pmatrix}
				=
				\mtx{\gamma(X)&0\\0&0}_{p\times q},
				\]
				where
				\[
				\gamma(X)=
				\begin{pmatrix}
					0&-\lambda(X)&-\delta_1(X)\\
					\lambda(X)&0&-\delta_2(X)\\
					\delta_1(X)&\delta_2(X)&0
				\end{pmatrix}\in\Alt_3(\F).
				\]
				All the matrices used in these transformations are invertible. Thus $\phi$ has form (6).
				\qedhere
			\end{enumerate}
			
		\end{enumerate}
	\end{proof}
	
	\subsection{Proof of the classification of rank-two preservers}
	
	\begin{proof}[Proof of Corollary \ref{thm: Mmn rank 2 preserver}.] 
		Since $|\F|\ge5>2+2$, Theorem~\ref{thm: preserving to decreasing} shows that $\phi$ is rank-two nonincreasing. Thus Theorem~\ref{thm: M_mn rank 2 nonincreasing} applies.
		
		We determine which of the six forms preserve rank
		two. For forms (1) and (2), the two constant factors must have full
		column and full row rank, respectively; otherwise a rank-two matrix
		can be chosen with a nonzero column or row vector in the relevant
		kernel. For form (3),
		\[
		\rank\bigl(\delta(A)V^\top\bigr)=2
		\]
		for every rank-two $A$ exactly when $\rank V=2$ and
		$\rank\delta(A)=2$ for every such $A$. The transpose argument gives
		the condition in form (4). In form (5), the sum
		$u\rho(A)^\top+\tau(A)v^\top$ has rank two exactly when
		$\tau(A)\notin\F u$ and $\rho(A)\notin\F v$. 
		
		For form (6),
		$\rank\phi(A)=\rank\gamma(A)$, and every nonzero $3\times3$
		alternating matrix in the displayed form has rank two. Hence rank-two preservation is equivalent to $\ker\gamma$ containing no rank-two matrix. Note that
		\begin{equation}\label{codim ker gamma}
			mn-\dim\ker\gamma=\dim\Ima\gamma\le 3.
		\end{equation}
		By \cite[Corollary 4]{PAZZIS2011}, if $\max\{m,n\}\ge 4,$ $\ker\gamma$ always contains a rank-two matrix; thus $\gamma$ cannot be a rank-two preserver. 
		Moreover,  $\min\{m,n\}\ge 2$ by assumption. It remains to consider the smaller source spaces.
		\begin{itemize}
			\item $(m,n)=(2,2).$ 
			Since $\dim \ker \gamma \ge 4 - 3 = 1$, $\ker \gamma$ is a nonzero subspace of upper rank $1$.  By \cite[Lemma~2]{AtkinsonLloyd1981}, every matrix space of upper
			rank at most $1$ is equivalent to a subspace supported on one row
			or one column. In particular, its dimension is at most $\max\{m,n\}$.
			We have $\dim\ker\gamma\le 2$, and $\ker\gamma$ is exactly one of the three types in the case 6a). 
			\item $(m,n)=(3,2).$ Again, $\ker\gamma$ is a subspace of upper rank one. Moreover, \eqref{codim ker gamma} shows that $\dim\ker\gamma\ge 3$.  So \cite[Lemma~2]{AtkinsonLloyd1981} implies that there is a nonzero vector $w\in\F^2$ such that
			\[
			\ker\gamma =\{uw^\top\mid u\in\F^3\}.
			\]
			Let $v\in\F^2$ be a nonzero vector such that $w^\top v=0.$ 
			The surjective linear map $A\mapsto Av$ has the same kernel as $\gamma$. 
			There is a linear isomorphism 
			\[\varphi: \{Av\mid A\in\M_{3,2}(\F)\}\to \Alt_3(\F)\] such that  
			$\gamma(A)=\varphi(Av).$
			Hence we get the form in 6b) for $(m,n)=(3,2).$
			
			\item $(m,n)=(2,3).$ The transpose argument gives the form in 6b) for $(m,n)=(2,3).$
			\item $(m,n)=(3,3).$ By \eqref{codim ker gamma}, $\dim\ker\gamma\ge6$.
			Let
			\[V:=\{A\in\ker\gamma\mid\det(A)=0\}.\]
			Every matrix in $V$ has rank at most one. Moreover, if $A,B\in V$, then $A+B\in\ker\gamma$ has rank at most two; since the kernel contains no rank-two matrix, $A+B\in V$. Thus $V$ is a subspace of upper rank one, and $\dim V\le3$ by \cite[Lemma~2]{AtkinsonLloyd1981}.
			Choose a nonzero $A\in V$ and a vector $v\in\F^3$ with $Av\ne0$. Such an $A$ exists: otherwise every nonzero element of $\ker\gamma$ would be invertible, and evaluation at any nonzero $v$ would inject the space of dimension at least six into $\F^3$. Now the evaluation map
			\[
			\ker\gamma\longrightarrow\F^3,\qquad X\longmapsto Xv
			\]
			has a kernel of dimension at least three. Every matrix in that kernel is singular and hence belongs to $V$. But the restriction of evaluation to $V$ is nonzero because $Av\ne0$, so its kernel has dimension at most two. This contradiction rules out $(m,n)=(3,3)$   in the stated class.
			\qedhere 
			
		\end{itemize}
	\end{proof}

	\section*{Acknowledgments}
	Ming-Cheng Tsai is supported by Taiwan NSTC grant 115-2115-M-027-001-MY2.
	
	\section*{Declaration of generative AI use}
	The authors used
	ChatGPT (OpenAI; model: GPT-6.1 Sol) to assist with language
	editing, manuscript organization, improving the exposition of portions of proofs, and the identification of  references.
	The authors independently reviewed and verified the mathematical
	arguments and cited references, revised the AI-assisted material,
	and confirmed the originality and accuracy of the final content.
	The authors have reviewed the tool's applicable terms of use
	and confirmed the suitability of its use for publication.
	The authors take full responsibility for the integrity and
	content of the manuscript.


\end{document}